\documentclass{article}
\usepackage{amsmath}
\usepackage{amsfonts}
\usepackage{amssymb}
\usepackage{graphicx}
\usepackage{hyperref}
\usepackage{mathabx}

\graphicspath{ {\string~/Documents/Images/} }
\providecommand{\U}[1]{\protect\rule{.1in}{.1in}}
\newtheorem{theorem}{Theorem}[section]

\newtheorem{conjecture}{Conjecture}[section]
\newtheorem{corollary}{Corollary}[section]

\newtheorem{definition}{Definition}[section]
\newtheorem{example}{Example}[section]

\newtheorem{lemma}{Lemma}[section]

\newtheorem{proposition}{Proposition}[section]
\newtheorem{remark}{Remark}[section]

\newenvironment{proof}[1][Proof]{\textbf{#1.} }{\ \rule{0.5em}{0.5em}}
\input{tcilatex}
\begin{document}

\title{Edge-Transitive Graphs}
\author{Heather A. Newman \\
Department of Mathematics\\
Princeton University \and Hector Miranda \\
School of Mathematical Sciences\\
Rochester Institute of Technology \and Darren A. Narayan \\
School of Mathematical Sciences\\
Rochester Institute of Technology}
\maketitle

\begin{abstract}
A graph is said to be edge-transitive if its automorphism group acts
transitively on its edges. It is known that edge-transitive graphs are
either vertex-transitive or bipartite. In this paper we present a complete
classification of all connected edge-transitive graphs on less than or equal
to $20$ vertices. We then present a construction for an infinite family of
edge-transitive bipartite graphs, and use this construction to show that
there exists a non-trivial bipartite subgraph of $K_{m,n}$ that is connected
and edge-transitive whenever $gcd(m,n)>2$. Additionally, we investigate
necessary and sufficient conditions for edge transitivity of connected $%
(r,2) $ biregular subgraphs of $K_{m,n}$, as well as for uniqueness, and use
these results to address the case of $gcd(m,n)=2$. We then present infinite
families of edge-transitive graphs among vertex-transitive graphs, including
several classes of circulant graphs. In particular, we present necessary
conditions and sufficient conditions for edge-transitivity of certain
circulant graphs.
\end{abstract}

\section{Introduction}

A graph is vertex-transitive (edge-transitive) if its automorphism group
acts transitively on its vertex (edge) set. We note the alternative
definition given by Andersen et al. \cite{Andersen}.

\begin{theorem}[Andersen, Ding, Sabidussi, and Vestergaard]
A finite simple graph $G$ is edge-transitive if and only if $G-e_{1}\cong
G-e_{2}$ for all pairs of edges $e_{1}$ and $e_{2}$.
\end{theorem}

We also mention the following well-known result.

\begin{proposition}
\label{p1} If $G$ is an edge-transitive graph, then $G$ is either
vertex-transitive or bipartite; in the latter case, vertices in a given part
belong to the same orbit of the automorphism group of $G$ on vertices.
\end{proposition}

Given a graph $G$ we will denote its vertex set by $V(G)$ and edge set by $%
E(G)$. We will use $K_{n}$ to denote the complete graph with $n$ vertices,
and $K_{m,n}$ to denote the complete bipartite graph with $m$ vertices in
one part and $n$ in the other. The path on $n$ vertices will be denoted $%
P_{n}$ and the cycle on $n$ vertices with $C_{n}$. The complement of a graph 
$G$ will be denoted $\overline{G}$. We continue with series of definitions
with elementary graph theory properties.

\begin{definition}
A graph is regular if all of its vertices have the same degree. A
(connected) bipartite graph is said to be biregular if all vertices on the
same side of the bipartition all have the same degree. Particularly, we
refer to a bipartite graph with parts of size $m$ and $n$ as an $(r,s)$
biregular subgraph of $K_{m,n}$ if the $m$ vertices in the same part each
have degree $r$ and the $n$ vertices in the same part have degree $s$.
\end{definition}

It follows from Proposition \ref{p1} that bipartite edge-transitive graphs
are biregular.

\begin{definition}
Given a group $G$ and generating set $S$, a Cayley graph\ $\Gamma (G,S)$ is
a graph with vertex set $V(\Gamma )$ and edge set $E(\Gamma )=$ $\{\{x,y\}$
: $x,y$ in $V(\Gamma )$, there exists an integer $s$ in $S$ : $y=xs\}$.
\end{definition}

It is known that all Cayley graphs are vertex-transitive. Next we recall a
specialized class of Cayley graphs known as circulant graphs.

\begin{definition}
A circulant graph $C_{n}(L)$ is a graph on vertices $v_{1},v_{2},...,v_{n}$
where each $v_{i}$ is adjacent to $v_{(i+j)(\func{mod}n)}$ and $v_{(i-j)(%
\func{mod}n)}$ for each $j$ in a list $L$. Algebraically, circulant graphs
are Cayley graphs of finite cyclic groups. For a list $L$ containing $m$
items, we refer to $C_{n}(L)$ as an $m$-circulant. We say an edge $e$ is a
chord of length $k$ when $e=v_{i}v_{j}$, $|i-j|\equiv k$ mod $n$.
\end{definition}

In our next definition we present another family of vertex-transitive graphs.

\begin{definition}
A wreath graph, denoted $W(n,k)$, has n sets of $k$ vertices each, arranged
in a circle where every vertex in set $i$ is adjacent to every vertex in
bunches $i+1$ and $i-1$. More precisely, its vertex set is $\mathbb{Z}_{n}$ $%
\times$ $\mathbb{Z}_{k}$ and its edge set consists of all pairs of the form $%
\{(i,r),(i+1,s)\}$.
\end{definition}

It was proved by Onkey \cite{Onkey} that all wreath graphs are
edge-transitive. We next recall the definition of the line graph which we
use later to show that certain graph families are edge-transitive.

\begin{definition}
Given a graph $G$, the line graph $L(G)$ is a graph where $V(L(G))=E(G)$ and
two vertices in $V(L(G))$ are adjacent in $L(G)$ if and only if their
corresponding edges are incident in $G$.
\end{definition}

Finally we recall the operation of the Cartesian product of graphs.

\begin{definition}
Given two graphs $H$ and $K$, with vertex sets $V(H)$ and $V(K)$, the
Cartesian product $G=H\times K$ is a graph where $V(G)=\{(u_{i},v_{j})$
where $u_{i}\in V(H)$ and $v_{j}\in V(K)\}$, and $E(G)=\left\{
(u_{i},v_{j}),(u_{k},v_{l})\right\} $ if and only if $i=k$ and $v_{j}$ and $%
v_{l}$ are adjacent in $K$ or $j=l$ and $u_{i}$ and $u_{k}$ are adjacent in $%
H$.
\end{definition}

The properties vertex-transitive and edge-transitive are distinct. This is
clear with the following examples:

\begin{itemize}
\item $K_{n},$ $n\geq2$ is both vertex-transitive and edge-transitive.

\item $C_{n}(1,2)$, $n\geq6$ is vertex-transitive, but not edge-transitive.

\item $K_{1,n-1}$ is not vertex-transitive, but is edge-transitive.

\item $P_{n}$, $n\geq4$ is neither vertex-transitive nor edge-transitive.
\end{itemize}

However the two properties are linked, as evident from the following
proposition following from the results of Whitney \cite{Whitney} and
Sabidussi \cite{Sabidussi}.

\begin{proposition}
A connected graph is edge-transitive if and only if its line graph is
vertex-transitive.
\end{proposition}

Note, however, that a graph may not be the line graph of some original
graph. For example, $K_{1,3}\times C_{4}$ is vertex-transitive, but it
follows by a theorem of Beineke \cite{Beineke} that this graph is not the
line graph of some graph.

We used the databases from Brendan McKay \cite{McKay} to obtain all
connected edge-transitive graphs on $20$ vertices or less. We then reported
the number of edge-transitive graphs up to $20$ to the Online Encyclopedia
of Integer Sequences, and they are listed under sequence \#A095424. The full
classification of these graphs is given in the appendix of this paper. We
can extrapolate much from this data and these results are presented in this
paper. It was recently brought to our attention that Marston Conder and
Gabriel Verret independently determined the edge-sets of the connected
edge-transitive bipartite graphs on up to 63 vertices \cite{Conder} using
the Magma system, and a complete list of all connected edge-transitive
graphs on up to 47 vertices \cite{Conder47} with their edge sets \cite%
{Conder472}. These graphs do not appear to be named. In our paper, we
carefully named the graphs, allowing us to generalize some cases to infinite
families of graphs.

We note the following graph families are edge-transitive: $K_{n}$, $n\geq2$, 
$C_{n}$, $n\geq3$, $K_{n,n}$ minus a perfect matching, $K_{2n}$ minus a
perfect matching, and all complete bipartite graphs $K_{t,n-t}$, $1\leq
t\leq\left\lfloor \frac{n}{2}\right\rfloor $. Wreaths \cite{Onkey} and
Kneser graphs \cite{Royle} are also edge-transitive. Besides these
predictable and apparent cases, we can identify other infinite families of
edge-transitive graphs, using the data up through 20 vertices.

Recalling that edge-transitive graphs are either vertex-transitive or
biregular bipartite, we have the following motivating questions:

\begin{itemize}
\item For which values of $m,n,r,s$ such that $mr=ns$ does there exist a
(connected) $(r,s)$ biregular subgraph of $K_{m,n}$ that is edge-transitive?
For example, why do there exist \textit{two} edge-transitive and connected $%
(4,2)$ biregular subgraphs of $K_{5,10}$, but no edge-transitive and
connected $(5,2)$ biregular subgraph of $K_{4,10}$? (These graphs are found
in the appendix).

\item Can we identify infinite classes of edge-transitive graphs among
vertex-transitive graphs, particularly circulants? Furthermore, can we
identify necessary and sufficient conditions for edge-transitivity of
certain circulants?
\end{itemize}

\section{Edge-Transitive, Connected Bipartite Graphs}

We begin by addressing the first motivating question from the previous
section. Given positive integers $m$ and $n$, we first describe
formulaically which values of $r$ and $s$ are possible for an $(r,s)$
biregular subgraph of $K_{m,n}$. Note that if $gcd(m,n)=1$, the only
biregular subgraph of $K_{m,n}$ is $K_{m,n}$.

\begin{proposition}
An $(r,s)$ biregular subgraph of $K_{m,n}$ satisfies the following: 
\begin{equation*}
mr=ns
\end{equation*}%
\begin{equation*}
r=\frac{n}{gcd(m,n)}k,k=1,2,\dots ,gcd(m,n)
\end{equation*}
\end{proposition}

\begin{proof}
We know that $s=\frac{mr}{n}=\frac{\frac{m}{gcd(m,n)}}{\frac{n}{gcd(m,n)}}r$
and since $gcd\left( \frac{m}{gcd(m,n)},\frac{n}{gcd(m,n)}\right) =1$, $r$
is a multiple of $\frac{n}{gcd(m,n)}$ (and is less than or equal to $n$).
\end{proof}

\begin{corollary}
If $gcd(m,n)=2$, there are only two possible pairs $(r,s)$, namely, $%
(r,s)=\left( \frac{n}{2},\frac{m}{2}\right) $ and $(r,s)=(n,m)$. The latter
case is the complete bipartite graph $K_{m,n}$.
\end{corollary}

We now introduce a construction for generating non-trivial edge-transitive
(connected) bipartite subgraphs of $K_{m,n}$ for $gcd(m,n)>2$. This
construction involves a process of extending a non-trivial edge-transitive
(connected) bipartite graph to a larger one, which we describe in the
following lemma.

\begin{lemma}
Let $G$ be an edge-transitive (connected) $(r,s)$ biregular subgraph of $%
K_{m,n}$. Then, for any positive integers $a, b$, $G$ can be extended to an
edge-transitive (connected) $(ra,sb)$ biregular subgraph of $K_{mb, na}$.
\end{lemma}

\begin{proof}
It suffices to show that, letting $G$ be a (connected) edge-transitive $%
(r,s) $ biregular subgraph of $K_{m,n}$, we can build a (connected)
edge-transitive graph $H$ that is an $(r,2s)$ biregular subgraph of $%
K_{2m,n} $. Let $G$ consist of partite sets $A,B$ where $A=\{a_{1},a_{2},%
\dots ,a_{m}\}$ and $B=\{b_{1},b_{2},\dots ,b_{n}\}$. Now create the set $%
A^{\prime }=\{a_{1},a_{2},\dots ,a_{m},a_{1}^{\prime },a_{2}^{\prime },\dots
,a_{m}^{\prime }\}$ and create a graph $H$ with partite sets $A^{\prime }$
and $B$ as follows. For each $a_{i}$, let $N_{H}(a_{i})=N_{G}(a_{i})$. For
each $a_{i}^{\prime }$, let $N_{H}(a_{i}^{\prime })=N_{G}(a_{i})$. Then by
construction, $H$ is a (connected) $(r,2s)$ biregular subgraph of $K_{2m,n}$%
. Since $G$ is edge-transitive, $H$ is edge-transitive by construction.
\end{proof}

\begin{theorem}
\label{21} Let $gcd(m,n) > 2$. Then there exists an edge-transitive
(connected) subgraph of $K_{m,n}$ which is not complete.
\end{theorem}

\begin{proof}
We appeal to the construction in the preceding lemma, and consider the
following two cases:

\noindent \newline
\underline{Case 1:} $m|n$ \newline
Then $n=mk$ for some positive integer $k$. Let $G$ be the graph that results
from removing a perfect matching from $K_{m,m}$. Then $G$ is connected,
biregular, and edge-transitive but not complete. Repeating the construction
in the lemma $k-1$ times, we obtain a subgraph of $K_{m,mk}=K_{m,n}$ that is
connected, biregular, edge-transitive, and not complete.

\noindent \newline
\underline{Case 2:} $m\nmid n$ \newline
Let $l=gcd(m,n)$ and $m=k_{1}l,n=k_{2}l$. Let $G$ be the graph that results
from removing a perfect matching from $K_{l,l}$. Then $G$ is connected,
biregular, and edge-transitive but not complete. Then, following the
construction in the lemma, increase the left partite set by $l$ vertices $%
k_{1}-1$ times and the right partite set by $l$ vertices $k_{2}-1$ times.
The resulting graph will be a connected, biregular, and edge-transitive
subgraph of $K_{k_{1}l,k_{2}l}=K_{m,n}$ but not complete.
\end{proof}

\pagebreak \FRAME{dtbpFU}{6.0883in}{2.0946in}{0pt}{\Qcb{Figure 1. An example
of the construction in the theorem, with vertices drawn in the same color
being vertices that are connected to the graph in the same way.}}{}{Figure}{%
\special{language "Scientific Word";type "GRAPHIC";maintain-aspect-ratio
TRUE;display "USEDEF";valid_file "T";width 6.0883in;height 2.0946in;depth
0pt;original-width 6.0277in;original-height 2.0557in;cropleft "0";croptop
"1";cropright "1";cropbottom "0";tempfilename
'OW9TAU00.wmf';tempfile-properties "XPR";}}

\begin{remark}
\noindent Theorem \ref{21} gives rise to the following
observations/questions:

\begin{itemize}
\item When $gcd(m,n)=1$, the only possible (connected) biregular subgraph is
the complete graph $K_{m,n}$.

\item When $gcd(m,n)=2$, the method fails because the only connected,
biregular subgraph of $K_{2,2}$ is $K_{2,2}$, and we seek a non-complete
bipartite graph.

\item When $gcd(m,n)=2$, under what additional conditions does the theorem
still hold?
\end{itemize}
\end{remark}

\subsection{Edge-transitive (connected) $(r,2)$ biregular subgraphs of $%
K_{m,n}$}

Before we address the case of $gcd(m,n)=2$, we introduce some new notation,
definitions, and constructions.

Let $G$ be a (connected) $(r,2)$ biregular subgraph of $K_{m,n}$. Then the
line graph $L(G)$ of $G$ consists of $m$ sets of $r$ vertices $%
V_{1},V_{2},\dots ,V_{m}$, with each of the sets $V_{i}$ of $r$ vertices in $%
L(G)$ corresponding to the $r$ edges incident to one of the $m$ vertices in
the left partite set. All of the vertices in a given set $V_{i}$ are
adjacent to each other. Additionally, each vertex in a given set $V_{i}$ is
adjacent to \textit{exactly one} vertex in another set $V_{j}$, $i\neq j$,
since $G$ is $(r,s)$ biregular with $s=2$.

Consider the vertex set $V_{i}$. We use the notation $P(V_{i})=[a_{2},a_{3},%
\dots ,a_{m}]$ to describe the partitioning of the $r$ edges that have
exactly one endpoint in $V_{i}$ among the $m-1$ remaining sets $V_{j}$. For
example, if $m=4$ and $r=4$, the notation $P(V_{1})=[2,2,0]$ means that,
without loss of generality, of the four edges with exactly one endpoint in $%
V_{1}$, there are $a_{2}=2$ edges that are incident with vertices in $V_{2}$%
, $a_{3}=2$ edges that are incident with vertices in $V_{3}$, and $a_{4}=0$
edges that are incident with a vertex in $V_{4}$.

The order of the coordinates does not matter. That is, $[3,1,0]$ is the same
as $[1,0,3]$.

\begin{definition}
Let $A_{i}\subset V_{i}$ be a set of vertices corresponding to a part of $%
P(V_{i})$ and $A_{j}\subset V_{j}$ be a set of vertices corresponding to a
part of $P(V_{j})$, $i\neq j$, and suppose $|A_{i}|=|A_{j}|$. We say that $%
A_{i}$ is \textbf{one-to-one connected} to $A_{j}$ if each vertex in $A_{i}$
is adjacent to exactly one vertex in $A_{j}$ (and vice versa).
\end{definition}

\begin{definition}
We say the set $V_{i}$ is \textbf{uniformly partitioned} if the \textit{%
nonzero} coordinates of $P(V_{i})$ take on exactly one value, and we call $%
P(V_{i})$ a uniform partition.
\end{definition}

\noindent For instance, $[3,1,0]$ is not a uniform partition for $r=4$ among 
$m-1=3$ coordinates because there are two distinct values of the nonzero
coordinates, $3$ and $1$, but $[2,2,0]$ is a uniform partition because $2$
is the only nonzero coordinate.

\FRAME{dtbpFU}{5.578in}{2.4716in}{0pt}{\Qcb{Figure 2. $L(G)$ and $R(L(G))$
for a $(6,2)$ biregular subgraph of $K_{4,12}$. $%
P(V_{1})=P(V_{2})=P(V_{3})=P(V_{4})=[3,3,0]$, which is a uniform partition.}%
}{}{Figure}{\special{language "Scientific Word";type
"GRAPHIC";maintain-aspect-ratio TRUE;display "USEDEF";valid_file "T";width
5.578in;height 2.4716in;depth 0pt;original-width 5.521in;original-height
2.431in;cropleft "0";croptop "1";cropright "1";cropbottom "0";tempfilename
'OW9TAU01.wmf';tempfile-properties "XPR";}}

\begin{definition}
Suppose all the $V_{i}$ are uniformly partitioned in the same way, that is,
each $P(V_{i})$ is a uniform partition and $%
P(V_{i})=P(V_{j})=[d,d,...,d,0,0] $ for all $i,j$ and some positive integer $%
d$. (Note that, in the context of the dimensions $r$ and $m$, $d$ divides $r$
and $\frac{r}{d}<m$.) We say a graph $R(L(G))$ is a \textbf{reduction graph}
of the line graph $L(G)$ if it is constructed from $L(G)$ in the following
way:

\begin{itemize}
\item $R(L(G))$ consists of $m$ sets of $\frac{r}{d}$ vertices $V_{1}
^{\prime}, V_{2}^{\prime}, \dots, V_{m}^{\prime}$, and for every $i$, all of
the vertices in $V_{i}^{\prime}$ are adjacent to each other.

\item For $i \neq j$, let $A_{i} \subset V_{i}$ be a set of vertices
corresponding to a part of $P(V_{i})$, and let $A_{j} \subset V_{j}$ be a
set of vertices corresponding to a part of $P(V_{j})$. Then a one-to-one
connection (defined earlier) between $A_{i}$ and $A_{j}$ in $L(G)$
corresponds to a vertex in $V_{i}^{\prime}$ being adjacent to exactly one
vertex in $V_{j}^{\prime}$.

\item If $A_{i}$ and $B_{i}$ are vertex sets corresponding to two distinct
parts of $P(V_{i})$, then $A_{i}$ is one-to-one connected to a subset of
vertices in $V_{j}$ and $B_{i}$ is one-to-one connected to a subset of
vertices in $V_{k}$ where \textbf{$j \neq k$}. Therefore, in $R(L(G))$, two
vertices in a given set $V_{i}^{\prime}$ must not be adjacent to vertices in
the same set (other than within $V_{i}$ itself).
\end{itemize}
\end{definition}

\begin{theorem}
\label{r1}Let $mr=2n$. If $r$ is even, then there exists an edge-transitive
(connected) $(r,2)$ biregular subgraph of $K_{m,n}$.
\end{theorem}

\begin{proof}
Since $r$ is even, partition the $r$ vertices in each $V_{i}$ by $P(V_{i})=[%
\frac{r}{2},\frac{r}{2},0,\dots ,0]$ (where of course $P(V_{i})$ has $m-1$
coordinates), and denote the two parts of $\frac{r}{2}$ vertices
corresponding to this partition $A_{i}$ and $B_{i}$. Let $A_{i}$ be
one-to-one connected to $B_{i-1}$ when $i>1$ and $B_{m}$ when $i=1$. Let $%
B_{i}$ be one-to-one connected to $A_{i+1}$ when $i<m$ and $B_{1}$ when $i=m$%
. The graph $L(G)$ is therefore vertex-transitive, meaning that the
corresponding graph $G$ is edge-transitive.
\end{proof}

\begin{remark}
The graph proposed in the theorem is not necessarily unique. For example,
there exist two edge-transitive and connected $(6,2)$ biregular subgraphs of 
$K_{4,12}$, one with a line graph corresponding to the partition $[3,3,0]$,
which is the construction used in the proof of the theorem, and one with a
line graph corresponding to the partition $[2,2,2]$.
\end{remark}

\FRAME{dtbpFU}{4.8456in}{2.0807in}{0pt}{\Qcb{Figure 3. Line graph
representations of $(6,2)$ biregular subgraphs of $K_{4,12}$. The left one
is the construction used.}}{}{Figure}{\special{language "Scientific
Word";type "GRAPHIC";maintain-aspect-ratio TRUE;display "USEDEF";valid_file
"T";width 4.8456in;height 2.0807in;depth 0pt;original-width
4.7919in;original-height 2.0418in;cropleft "0";croptop "1";cropright
"1";cropbottom "0";tempfilename 'OW9TAU02.wmf';tempfile-properties "XPR";}}

\begin{theorem}
\label{r2}Let $G$ be an $(r,2)$ biregular subgraph of $K_{m,n}$. If $G$ is
edge-transitive and connected, then, using the notation above, every $V_{i}$
must be uniformly partitioned, that is the \textit{nonzero} coordinates $%
a_{i}$ in $P(V_{i})=[a_{2},a_{3},\dots ,a_{m}]$ take on exactly one value.
\end{theorem}

\begin{proof}
We proceed by contradiction, that is, we assume that $G$ is edge-transitive
and that there exists a partition $P(V_{i})$ that contains two distinct
nonzero coordinates $a$ and $b$. Since $G$ is edge-transitive, $L(G)$ must
be vertex-transitive. Let $A \subset V_{i}$ be a vertex set corresponding to
a part of $P(V_{i})$ of size $a$ and $B \subset V_{i}$ be a vertex set
corresponding to a part of $P(V_{i})$ of size $b$. Moreover, choose any
vertex $v_{a} \in A$ and $v_{b} \in B$. Then $v_{a}$ has $a-1$ neighbors in $%
V_{i}$ which are adjacent to vertices in some $V_{j}$ $(j\neq i)$ and $b$
neighbors in $V_{i}$ which are adjacent to vertices in some $V_{k}$ $(k \neq
i)$, such that $V_{j} \neq V_{k}$. Similarly, $v_{b}$ has $b-1$ neighbors in 
$V_{i}$ which are adjacent to vertices in $V_{k}$ and $a$ neighbors in $%
V_{i} $ which are adjacent to vertices in $V_{j}$. Since $a-1\neq b-1$, $%
v_{a}$ and $v_{b}$ are distinguishable from each other, meaning that $L(G)$
is not vertex-transitive, and so $G$ is not edge-transitive, which is a
contradiction (Figure 4).
\end{proof}

\begin{remark}
The condition in the theorem is necessary but not sufficient. That is not
every uniform partition has a realization in the line graph.
\end{remark}

\FRAME{dtbpFU}{2.7233in}{2.3592in}{0pt}{\Qcb{Figure 4. A vertex in set $A$
can be distinguished from a vertex in set $B$, so $L(G)$ is not
vertex-transitive, implying that $G$ is not edge-transitive.}}{}{Figure}{%
\special{language "Scientific Word";type "GRAPHIC";maintain-aspect-ratio
TRUE;display "USEDEF";valid_file "T";width 2.7233in;height 2.3592in;depth
0pt;original-width 2.6809in;original-height 2.3194in;cropleft "0";croptop
"1";cropright "1";cropbottom "0";tempfilename
'OW9TAU03.wmf';tempfile-properties "XPR";}}

\begin{example}
As a counterexample, we consider that a $[2,2,2,0]$ partitioning for a line
graph $L(G)$ of a $(6,2)$ biregular subgraph $G$ of $K_{5,15}$ has no
realization. For, assuming by contradiction that such a realization $L(G)$
exists, $R(L(G))$ has $m=5$ sets $V_{1}^{\prime },V_{2}^{\prime },\dots
,V_{5}^{\prime }$ of $r=3$ vertices. But a construction $R(L(G))$ satisfying
the conditions of reduction graphs (mentioned earlier) is impossible, since $%
R(L(G))$ would be the line graph of a $(3,2)$ biregular subgraph of $K_{5,x}$
for some $x$. But $5\cdot 3=2x$ does not have an integer solution $x$, which
is a contradiction.
\end{example}

\begin{corollary}
Let $\gcd (4,n)=2$ and $6\nmid n$. Then there does not exist an
edge-transitive (connected) $\left( \frac{n}{2},2\right) $ biregular
subgraph of $K_{4,n}$.
\end{corollary}

\begin{proof}
Since $\gcd (4,n)=2$, $r=\frac{n}{2}$ is not even. Furthermore, since $6$
does not divide $n$, $3$ does not divide $r=\frac{n}{2}$. By the theorem,
the only eligible partitions are $\left[ \frac{r}{2},\frac{r}{2},0\right] $
and $\left[ \frac{r}{3},\frac{r}{3},\frac{r}{3}\right] $. But since $r$ is
not even, the first partition cannot occur, and since $r$ is not divisible
by $3$, the second partition cannot occur.
\end{proof}

\begin{corollary}
\label{c23} Let $\gcd (4,n)=2$ and $6|n$. Then there exists a \textbf{unique}
edge-transitive (connected) $(\frac{n}{2},2)$ biregular subgraph of $K_{4,n}$%
.
\end{corollary}

\begin{proof}
Since $r=\frac{n}{2}$ is divisible by $3$, of the two eligible partitions $%
\left[ \frac{r}{2},\frac{r}{2},0\right] $ and $\left[ \frac{r}{3},\frac{r}{3}%
,\frac{r}{3}\right] $ discussed in the previous proof, the partition $\left[ 
\frac{r}{3},\frac{r}{3},\frac{r}{3}\right] $ can occur (but the partition $%
\left[ \frac{r}{2},\frac{r}{2},0\right] $ cannot, since again $n$ is not
divisible by $4$). Note that, since there are no $0$ values in the
partition, any line graph that realizes this partition must be unique,
making the original graph unique as well. There exists a vertex-transitive
line graph that realizes this partition when $n=6$, meaning there exists a $%
(3,2)$ biregular subgraph $G$ of $K_{4,6}$ which is edge-transitive. Indeed,
the construction of $L(G)$ generalizes to a $(3k,2)$ biregular subgraph $%
G(k) $ of $K_{4,6k}$. This is because, if a reduction graph $R(L(G(k)))$
exists, it must be isomorphic to $L(G)$, so the existence of $L(G)$ implies
the existence of $R(L(G(k)))$ (Figure 5).
\end{proof}

\FRAME{dtbpFU}{4.6985in}{1.7876in}{0pt}{\Qcb{Figure 5. The line graph for
the unique edge-transitive (connected) $(6,2)$ subgraph of $K_{4,12}$
(left). All line graphs of an edge-transitive and connected $(3k,2)$
subgraph $G(k)$ of $K_{4,6k}$ reduce to the $k=1$ case (right).}}{}{Figure}{%
\special{language "Scientific Word";type "GRAPHIC";maintain-aspect-ratio
TRUE;display "USEDEF";valid_file "T";width 4.6985in;height 1.7876in;depth
0pt;original-width 4.6458in;original-height 1.7504in;cropleft "0";croptop
"1";cropright "1";cropbottom "0";tempfilename
'OW9TAU04.wmf';tempfile-properties "XPR";}}

\section{Edge-Transitive Graphs Among Vertex-Transitive Graphs}

Recall that some previously known infinite families of vertex-transitive
graphs are Wreath graphs and Kneser graphs. We identify an additional
infinite family of edge-transitive graphs that are vertex-transitive, stated
in terms of Cartesian products.

\begin{theorem}
\label{31} The graph $\overline{K_{m}\times K_{n}}$ is edge-transitive.
\end{theorem}

\begin{proof}
The graph $\overline{K_{m}\times K_{n}}$ is precisely the graph $\overline{%
L(K_{m,n})}$, that is, the complement of the line graph of $K_{m,n}$ \cite%
{wolfram}. First, we observe the structure of $L(K_{m,n})$. Let the partite
sets of $K_{m,n}$ be $A=\{a_{1},a_{2},\dots ,a_{m}\}$ and $%
B=\{b_{1},b_{2},\dots ,b_{n}\}$. The graph $L(K_{m,n})$ consists of $m$ sets
of $n$ vertices, which we denote $V_{1},V_{2},\dots ,V_{m}$. The $n$
vertices in each $V_{i}$ correspond to the edges incident to $a_{i}$ in the
graph of $K_{m,n}$. Specifically, $V_{i}=\{v_{i,1},v_{i,2},\dots ,v_{i,n}\}$
where $v_{i,k}$ corresponds to the edge $a_{i}b_{k}$ in the graph $K_{m,n}$.
By construction, all of the vertices in a given set $V_{i}$ are adjacent to
each other, since these vertices correspond to all edges incident to $a_{i}$
in $K_{m,n}$. Additionally, each $v_{i,k}$ is adjacent to $v_{j,k}$ for all $%
j\neq i$, since these vertices correspond to all edges incident to $b_{k}$
in $K_{m,n}$. This completes the construction of $L(K_{m,n})$. To construct $%
\overline{L(K_{m,n})}$, we retain the vertex sets $V_{1},\dots ,V_{m}$.
However, now we have an $m$-partite graph, since none of the edges in $V_{i}$
are connected to each other in $\overline{L(K_{m,n})}$. Each $v_{i,k}$ is
connected to $v_{j,l}$ for all $j\neq i$ and all $l\neq k$. In other words,
all possible edges of the $m$-partite graph exist except for edges of the
form $v_{i,k}v_{j,k}$. It is clear from this description that $\overline{%
L(K_{m,n})}$ is edge-transitive. This follows from the fact every vertex in
a given partite set is indistinguishable from every other vertex in that
set, and the fact that each partite set is indistinguishable from every
other partite set. Since $\overline{K_{m}\times K_{n}}=\overline{L(K_{m,n})}$%
, the result is proven (Figure 6).
\end{proof}

\FRAME{dtbpFU}{5.201in}{1.9899in}{0pt}{\Qcb{Figure 6. An example of the
construction in the proof of the theorem for $m=4,n=3$.}}{}{Figure}{\special%
{language "Scientific Word";type "GRAPHIC";maintain-aspect-ratio
TRUE;display "USEDEF";valid_file "T";width 5.201in;height 1.9899in;depth
0pt;original-width 5.1456in;original-height 1.951in;cropleft "0";croptop
"1";cropright "1";cropbottom "0";tempfilename
'OW9TAU05.wmf';tempfile-properties "XPR";}}

\subsection{Edge Transitivity of Circulants}

We begin with a specialized class of circulants, known as powers of cycles.

\begin{definition}
The $k$-th power of $C_{n}$ is denoted $C_{n}^{k}$ and has vertices $%
v_{1},v_{2},...,v_{n}$ and edges between $v_{i}$ and $v_{j}$ whenever $%
\left\vert j-i\right\vert \leq k$ $\func{mod}n$. \ \ \ \ \ \ \ \ \ \ \ \ \ \
\ \ \ \ \ \ \ \ \ \ \ \ \ \ \ \ \ \ \ \ \ \ \ \ \ \ \ \ \ \ \ \ \ \ \ \ \ \
\ \ \ \ \ \ \ \ \ \ \ \ \ \ \ \ \ \ \ \ \ \ \ \ \ \ \ \ \ \ \ \ \ \ \ \ \ \
\ \ \ \ \ \ \ \ \ \ \ \ \ \ \ \ \ \ \ \ \ \ \ \ \ \ \ \ \ \ \ \ \ \ \ \ \ \
\ \ \ \ \ \ \ \ \ \ \ \ \ \ \ \ \ \ \ \ \ \ \ \ \ \ \ \ \ \ \ \ \ \ \ \ \ \
\ \ \ \ \ \ \ \ \ \ \ \ \ \ \ \ \ \ \ \ \ \ \ \ \ \ \ \ \ \ \ \ \ \ \ \ \ \
\ \ \ \ \ \ \ \ \ \ \ \ \ \ \ \ \ \ \ \ \ \ \ \ \ \ \ \ \ \ \ \ \ \ \ \ \ \
\ \ \ \ \ \ \ \ \ \ \ \ \ \ \ \ \ \ \ \ \ \ \ \ \ \ \ \ \ \ \ \ \ \ \ \ \ \
\ \ \ \ \ \ \ \ \ \ \ \ \ \ \ \ \ \ \ \ \ \ \ \ \ \ \ \ \ \ \ \ \ \ \ \ \ \
\ \ \ \ \ \ \ \ \ \ \ \ \ \ \ \ \ \ \ \ \ \ \ \ \ \ \ \ \ \ \ \ \ \ \ \ \ \
\ \ \ \ \ \ \ \ \ \ \ \ \ \ \ \ \ \ \ \ \ \ \ \ \ \ \ \ \ \ \ \ \ \ \ \ \ \
\ \ \ \ \ \ \ \ \ \ \ \ \ \ \ \ \ \ \ \ \ \ \ \ \ \ \ \ \ \ \ \ \ \ \ \ \ \
\ \ \ \ \ \ \ \ \ \ \ \ \ \ \ \ \ \ \ \ \ \ \ \ \ \ \ \ \ \ \ \ \ \ \ \ \ \
\ \ \ \ \ \ \ \ \ \ \ \ \ \ \ \ \ \ \ \ \ \ \ \ \ \ \ \ \ \ \ \ \ \ \ \ \ \
\ \ \ \ \ \ \ \ \ \ \ \ \ \ \ \ \ \ \ \ \ \ \ \ \ \ \ \ \ \ \ \ \ \ \ \ \ \
\ \ \ \ \ \ \ \ \ \ \ \ \ \ \ \ \ \ \ \ \ \ \ \ \ \ \ \ \ \ \ \ \ \ \ \ \ \
\ \ \ \ \ \ \ \ \ \ \ \ \ \ \ \ \ \ \ \ \ \ \ \ \ \ \ \ \ \ \ \ \ \ \ \ \ \
\ \ \ \ \ \ \ \ \ \ \ \ \ \ \ \ \ \ \ \ \ \ \ \ \ \ \ \ \ \ \ \ \ \ \ \ \ \
\ \ \ \ \ \ \ \ \ \ \ \ \ \ \ \ \ \ \ \ \ \ \ \ \ \ \ \ \ \ \ \ \ \ \ \ \ \
\ \ \ \ \ \ \ \ \ \ \ \ \ \ \ \ \ \ \ \ \ \ \ \ \ \ \ \ \ \ \ \ \ \ \ \ \ \
\ \ \ \ \ \ \ \ \ \ \ \ \ \ \ \ \ \ \ \ \ \ \ \ \ \ \ \ \ \ \ \ \ \ \ \ \ \
\ \ \ \ \ \ \ \ \ \ \ \ \ \ \ \ \ \ \ \ \ \ \ \ \ \ \ \ \ \ \ \ \ \ \ \ \ \
\ \ \ \ \ \ \ \ \ \ \ \ \ \ \ \ \ \ \ \ \ \ \ \ \ \ \ \ \ \ \ \ \ \ \ \ \ \
\ \ \ \ \ \ \ \ \ \ \ \ \ \ \ \ \ \ \ \ \ \ \ \ \ \ \ \ \ \ \ \ \ \ \ \ \ \
\ \ \ \ \ \ \ \ \ \ \ \ \ \ \ \ \ \ \ \ \ \ \ \ \ \ \ \ \ \ \ \ \ \ \ \ \ \
\ \ \ \ \ \ \ \ \ \ \ \ \ \ \ \ \ \ \ \ \ \ \ \ \ \ \ \ \ 
\end{definition}

\begin{theorem}
The graph $C_{n}^{k}$ is edge transitive if and only if (i) $k\geq
\left\lfloor \frac{n}{2}\right\rfloor $ or (ii)\ if $n$ is even and $k=\frac{%
n}{2}-1$.
\end{theorem}

\begin{proof}
Since $C_{n}^{k}$ is complete when $k\geq \left\lfloor \frac{n}{2}%
\right\rfloor $, it follows immediately that these graphs edge-transitive.
If $n=2m$ then the graph $C_{2m}^{m-1}$ is a complete graph minus a perfect
matching, which is also edge transitive.

Next we will show that $C_{n}^{k}$ where $k<\left\lfloor \frac{n}{2}%
\right\rfloor $ is not edge-transitive by showing that different edges are
contained in a different number of triangles. The edge $\left(
v_{i},v_{i+1}\right) $ is contained in $2k-1$ triangles:\ $\left(
v_{i-k+1},v_{i,}v_{i+1}\right) $, $\left( v_{i-k+2},v_{i,}v_{i+1}\right) $,
\dots , $\left( v_{i-1},v_{i,}v_{i+1}\right) $, $\left(
v_{i},v_{i+1,}v_{i+2}\right) $, $\left( v_{i},v_{i+1,}v_{i+3}\right) $,
\dots , $\left( v_{i},v_{i+1,}v_{i+k+1}\right) $. However the edge $\left(
v_{i},v_{i+k}\right) $ is only contained in $k-1$ triangles $\left\{ \left(
v_{i},v_{j,}v_{i+k}\right) |i+1\leq j\leq i+k-1\right\} $.
\end{proof}

\indent\newline
We use arguments similar to those in the last proof to demonstrate a large
class of circulant graphs that are \textit{not} edge-transitive.

\begin{theorem}
Let $n\in \mathbf{N}$, $n\geq 5$ and $S=\{2,3,\dots ,\lceil {\frac{n}{2}}%
\rceil -1\}$. Then circulants of the following form are \textit{not}
edge-transitive: 
\begin{equation*}
C_{2n}(A),\mbox{ }A=\{1,i,i+1,n\}\cup S^{\prime }%
\mbox{ for any $S'
\subseteq S$ and any $i, i+1 \in S$}
\end{equation*}
\end{theorem}

\begin{proof}
Let the vertices of $C_{2n}(A)$ be labelled 1 through $2n$. The edge $%
e_{1}=\{1,n+1\}$ is never contained in a 3-cycle, because the longest chord
is of length $\lceil \frac{n}{2}\rceil -1$, so 1 and $n+1$ share no common
neighbors. However, $e_{2}=\{1,2\}$ is always contained in the 3-cycle $%
\{1,2,i+2\}$.
\end{proof}

\noindent\newline
We proceed to introduce four classes of 3-circulant graphs that are
edge-transitive, the first three of which we demonstrate to be
edge-transitive by constructing explicit isomorphisms. Note that the edges
of 3-circulants split up into at most three orbits, with chords of the same
length belonging to the same orbit, so to prove edge-transitivity, it
suffices to choose a chord of one length and show that it can be mapped to a
chord of the second length and to a chord of the third length.

\begin{theorem}
Graphs in the following classes are edge-transitive: \noindent \newline
\noindent \newline
Class 1: $C_{20+16(n-1)}(1,2+4n,1+8n)$, $n=1,2,3,\dots $ \noindent \newline
\noindent \newline
Class 2: $C_{28+16(n-1)}(1,2+4n,5+8n)$, $n=1,2,3,\dots $ \noindent \newline
\noindent \newline
Class 3: $C_{12+6(n-1)}(1,2n+1,2n+3)$, $n=1,2,3,\dots $ \noindent \newline
\noindent \newline
Class 4: $C_{9+3(n-1)}(1,1+n,3+n)$, $n=1,2,3,\dots $\noindent \newline
\end{theorem}

\begin{proof}
For classes 1, 2, and 3, we construct explicit isomorphisms:

\noindent \newline
\underline{Class 1:} \noindent \newline
It suffices to show that there exists an isomorphism $\phi _{1}$ mapping the
edge $e_{1}=\{1,2\}$ to the edge $e_{2}=\{1,3+4n\}$ and an isomorphism $\phi
_{2}$ mapping the edge $e_{1}=\{1,2\}$ to the edge $e_{3}=\{1,2+8n\}$. We
provide $\phi _{1}$ and $\phi _{2}$ explicitly:

\noindent \newline
$\phi _{1}:1\mapsto 1$ \noindent \newline
$\phi _{1}:2\mapsto 3+4n$ \noindent \newline
$\phi _{1}:3\mapsto 12n+6$ \noindent \newline
$\phi _{1}:4\mapsto 4$ \noindent \newline
$\phi _{1}:(i+4k)\mapsto \phi (i)+4k$ mod $16n+4$, $i=1,2,3,4$

\noindent \newline
To show that $\phi _{1}$ is indeed an isomorphism, we show that $N(\phi
_{1}(i))=\phi _{1}(N(i))$ for, without loss of generality, $i=1,2,3,4$:

\noindent \newline
$N(\phi _{1}(1))=N(1)=\{2,16n+4,4n+3,8n+2,12n+3,8n+4\}$ $=\phi _{1}(N(1))$

\noindent \newline
$N(2)=\{3,1,4n+4,8n+3,12n+4,8n+5\}$ \noindent \newline
$\phi _{1}(N(2))=\{12n+6,1,4n+4,4n+2,12n+4,8n+5\}=N(\phi _{1}(2))$

\noindent \newline
$N(3)=\{4,2,4n+5,8n+4,12n+5,8n+6\}$ \noindent \newline
$\phi _{1}(N(3))=\{4,4n+3,4n+5,8n+4,12n+5,12n+7\}=N(\phi _{1}(3))$

\noindent \newline
$N(\phi _{1}(4))=N(4)=\{5,3,4n+6,8n+5,12n+6,8n+7\}=\phi _{1}(N(4))$

\noindent \newline
$\phi _{2}:1\mapsto 8n+2$ \noindent \newline
$\phi _{2}:2\mapsto 1$ \noindent \newline
$\phi _{2}:3\mapsto 8n+4$ \noindent \newline
$\phi _{2}:4\mapsto 3$ \noindent \newline
$\phi _{2}:(i+4k)\mapsto \phi (i)+4k$ mod $16n+4$, $i=1,2,3,4$

\noindent \newline
To show that $\phi _{2}$ is indeed an isomorphism, we show that $N(\phi
_{2}(i))=\phi _{2}(N(i))$ for, without loss of generality, $i=1,2,3,4$:

\noindent \newline
$\phi _{2}(N(1))=\{1,16n+3,12n+4,8n+1,4n,8n+3\}=N(\phi _{2}(1))$

\noindent \newline
$\phi _{2}(N(2))=\{8n+4,8n+2,4n+3,16n+4,12n+3,2\}=N(\phi _{2}(2))$

\noindent \newline
$\phi _{2}(N(3))=\{3,1,12n+6,8n+3,4n+2,8n+5\}=N(\phi _{2}(3))$

\noindent \newline
$\phi _{2}(N(4))=\{8n+6,8n+4,4n+5,2,12n+5,4\}=N(\phi _{2}(4))$

\noindent \newline
\underline{Class 2:} \noindent \newline
It suffices to show that there exists an isomorphism $\phi _{1}$ mapping the
edge $e_{1}=\{1,2\}$ to the edge $e_{2}=\{1,3+4n\}$ and an isomorphism $\phi
_{2}$ mapping the edge $e_{1}=\{1,2\}$ to the edge $e_{3}=\{1,6+8n\}$. We
provide $\phi _{1}$ and $\phi _{2}$ explicitly:

\noindent \newline
$\phi _{1}:1\mapsto 4n+3$ \noindent \newline
$\phi _{1}:2\mapsto 1$ \noindent \newline
$\phi _{1}:3\mapsto 2$ \noindent \newline
$\phi _{1}:4\mapsto 12n+12$ \noindent \newline
$\phi _{1}:(i+4k)\mapsto \phi (i)+4k$ mod $16n+12$, $i=1,2,3,4$

\noindent \newline
To show that $\phi _{1}$ is indeed an isomorphism, we show that $N(\phi
_{1}(i))=\phi _{1}(N(i))$ for, without loss of generality, $i=1,2,3,4$:

\noindent \newline
$N(1)=\{2,16n+12,4n+3,8n+6,12n+11,8n+8\}$ \noindent \newline
$\phi _{1}(N(1))=\{1,12n+8,4n+2,8n+5,12n+10,4n+4\}=N(\phi _{1}(1))$

\noindent \newline
$N(2)=\{3,1,4n+4,8n+7,12n+12,8n+9\}$ \noindent \newline
$\phi _{1}(N(2))=\{2,4n+3,16n+12,8n+6,8n+8,12n+11\}=N(\phi _{1}(2))$

\noindent \newline
$N(3)=\{4,2,4n+5,8n+8,12n+13,8n+10\}$ \noindent \newline
$\phi _{1}(N(3))=\{12n+12,1,8n+7,4n+4,3,8n+9\}=N(\phi _{1}(3))$

\noindent \newline
$N(4)=\{5,3,4n+6,8n+9,12n+14,8n+11\}$ \noindent \newline
$\phi _{1}(N(4))=\{4n+7,2,4n+5,12n+11,12n+13,8n+10\}=N(\phi _{1}(4))$

\noindent \newline
$\phi _{2}:1\mapsto 8n+6$ \noindent \newline
$\phi _{2}:2\mapsto 1$ \noindent \newline
$\phi _{2}:3\mapsto 8n+8$ \noindent \newline
$\phi _{2}:4\mapsto 3$ \noindent \newline
$\phi _{2}:(i+4k)\mapsto \phi (i)+4k$ mod $16n+12$, $i=1,2,3,4$

\noindent \newline
To show that $\phi _{2}$ is indeed an isomorphism, we show that $N(\phi
_{2}(i))=\phi _{2}(N(i))$ for, without loss of generality, $i=1,2,3,4$:

\noindent \newline
$\phi _{2}(N(1))=\{1,16n+11,12n+8,8n+5,4n+4,8n+7\}$ $=N(\phi _{2}(1))$

\noindent \newline
$\phi _{2}(N(2))=\{8n+8,8n+6,4n+3,16n+12,12n+11,2\}$ $=N(\phi _{2}(2))$

\noindent \newline
$\phi _{2}(N(3))=\{3,1,12n+10,8n+7,4n+6,8n+9\}=N(\phi _{2}(3))$

\noindent \newline
$\phi _{2}(N(4))=\{8n+10,8n+8,4n+5,2,12n+13,4\}=N(\phi _{2}(4))$

\noindent \newline
\underline{Class 3:} \noindent \newline
It suffices to show that there exists an isomorphism $\phi _{1}$ mapping the
edge $e_{1}=\{1,2\}$ to the edge $e_{2}=\{1,2n+2\}$ and an isomorphism $\phi
_{2}$ mapping the edge $e_{1}=\{1,2n+2\}$ to the edge $e_{3}=\{1,2n+4\}$. We
provide $\phi _{1}$ and $\phi _{2}$ explicitly:

\noindent \newline
$\phi _{1}:1\mapsto 2n+2$ \noindent \newline
$\phi _{1}:2\mapsto 1$ \noindent \newline
$\phi _{1}:3\mapsto 2$ \noindent \newline
$\phi _{1}:4\mapsto 3$ \noindent \newline
$\phi _{1}:2n+2\mapsto 2n+1$ \noindent \newline
$\phi _{1}:(i+(2n+2)k)\mapsto \phi (i)+(2n+2)k$ mod $6n+6$, $i=1,2,3,\dots
,2n+2$

\noindent \newline
To show that $\phi _{1}$ is indeed an isomorphism, we show that $N(\phi
_{1}(i))=\phi _{1}(N(i))$ for, without loss of generality, $i=1,2,k,2n+2$,
where $k\in \{3,4,\dots ,2n+1\}$

\noindent \newline
$N(1)=\{2,6n+6,2n+2,2n+4,4n+6,4n+4\}$ \noindent \newline
$\phi _{1}(N(1))=\{1,6n+5,2n+1,2n+3,4n+5,4n+3\}=N(\phi _{1}(1))$

\noindent \newline
$N(2)=\{1,3,2n+3,2n+5,1-2n,-1-2n\}$ \noindent \newline
$\phi _{1}(N(2))=\{2n+2,2,4n+4,2n+4,4n+6,6n+6\}=N(\phi _{1}(2))$

\noindent \newline
$N(k)=\{k+1,k-1,2n+k+1,2n+k+3,k-2n-1,k-2n-3\}$ \noindent \newline
$\phi _{1}(N(k))=\{k,k-2,2n+k,2n+k+2,k-2n-2,k-2n-4\}=N(\phi _{1}(k))$

\noindent \newline
$N(2n+2)=\{2n+1,2n+3,4n+3,4n+5,1,6n+5\}$ \noindent \newline
$\phi _{1}(N(2n+2))=\{2n,4n+4,4n+2,6n+6,2n+2,6n+4\}=N(\phi _{1}(2n+2))$

\noindent \newline
$\phi _{2}:1\mapsto 2n+4$ \noindent \newline
$\phi _{2}:2\mapsto 3$ \noindent \newline
$\phi _{2}:3\mapsto 4$ \noindent \newline
\noindent $\phi _{2}:4\mapsto 5$

\noindent $\phi _{2}:2n+1\mapsto 2n+2$ \noindent \newline
$\phi _{2}:2n+2\mapsto 1$ \noindent \newline
$\phi _{2}:(i+(2n+2)k)\mapsto \phi (i)+(2n+2)k$ mod $6n+6$, $i=1,2,3,\dots
,2n+2$

\noindent \newline
To show that $\phi _{2}$ is indeed an isomorphism, we show that $N(\phi
_{2}(i))=\phi _{2}(N(i))$ for, without loss of generality, $%
i=1,2,k,2n+1,2n+2 $, $k\in \{3,4,\dots ,2n\}$:

\noindent \newline
$\phi _{2}(N(1))=\{3,4n+5,1,2n+5,4n+7,2n+3\}=N(\phi _{2}(1))$

\noindent \newline
$\phi _{2}(N(2))=\{2n+4,4,2,2n+6,4n+8,4n+6\}=N(\phi _{2}(2))$

\noindent \newline
$\phi _{2}(N(k))=\{k,k+2,2n+k+2,2n+k+4,k-2n,k-2n-2\}$ $=N(\phi _{2}(k))$

\noindent \newline
$\phi _{2}(N(2n+1))=\{2n+1,1,4n+3,2n+3,4n+5,6n+5\}$ $=N(\phi _{2}(2n+1))$

\noindent \newline
$\phi _{2}(N(2n+2))=\{2n+2,2,4n+4,4n+6,2n+4,6n+6\}$ $=N(\phi _{2}(2n+2))$

\noindent \newline
\underline{Class 4:} \noindent \newline
Let $k=n+2$. Then by a result by Onkey \cite{Onkey} the graphs in Class 4
are of the form $C_{3k}(1,k-1,k+1)$, which is isomorphic to $W(k,3)$.
\end{proof}

\bigskip

The following theorem given by Wilson and Poto\u{c}nik \cite{Wilson}
presented a characterization of $4$-regular edge-transitive circulant graphs:

\begin{theorem}
If $G$ is an undirected 4-regular edge-transitive circulant graph with $n$
vertices, then either:

(1) $G\cong C_{n}(1,a)$ for some $a$ such that $a^{2}\equiv\pm 1(\func{mod}%
n) $, or

(2) $n$ is even, $n=2m$, and $G\cong C_{2m}(1,m+1)$.
\end{theorem}

From Onkey \cite{Onkey}, we have the following theorem for 6-regular
circulant graphs:

\begin{theorem}
Let $G=C_{n}(1,a,b)$, $a\neq b$, be a $6$-regular circulant. If the
following three conditions are \textit{all} satisfied, then $G$ is
edge-transitive:

\begin{enumerate}
\item $ab\equiv\pm1$ mod $n$

\item $a^{2}\equiv\pm b$ mod $n$

\item $b^{2}\equiv\pm a$ mod $n$
\end{enumerate}
\end{theorem}

We note the following corollaries can be obtained.

\begin{corollary}
3-circulants of the following forms are edge-transitive:

\begin{enumerate}
\item $C_{\frac{a^{3}\pm1}{d}(1,a,a^{2})}$ assuming $d|(a^{3}\pm1)$

\item $C_{\frac{a^{2}+a+1}{d}}(1,a,a+1)$ assuming $d|(a^{2}+a+1)$
\end{enumerate}
\end{corollary}

\begin{example}
Let $a\equiv 1$ mod $3$. Then $C_{\frac{a^{2}+a+1}{3}}(1,a,a+1)$ is
edge-transitive. (Note that if $a\equiv 1$ mod $3$, then $3\nmid a^{2}+a+1$).
\end{example}

\begin{example}
Let $a\equiv 2$ mod $7$ or $a\equiv 4$ mod $7$. Then $C_{\frac{a^{2}+a+1}{7}%
}(1,a,a+1)$ is edge-transitive. (Note that if $a\equiv 2$ mod $7$, and $%
a\equiv 4$ mod $7$ then $7|a^{2}+a+1$).
\end{example}

However we note that the conditions found in Theorem 3.6 are not necessary.
In Theorem 3.4 we presented infinite classes of graphs that do not satisfy
all three of the conditions from Theorem 3.6 but are still edge-transitive.
We believe in some cases that for a graph to be edge-transitive it may only
have to meet a set of weaker conditions.

\begin{conjecture}
Given $C_{n}(1,a,b)$ which is edge-transitive. Then at least one of the
following conditions is satisfied:

\begin{enumerate}
\item $ab\equiv \pm 1$ mod $n$ or

\item $a^{2}\equiv \pm 1$ mod $n$ or

\item $b^{2}\equiv\pm1$ mod $n$
\end{enumerate}
\end{conjecture}

\begin{conjecture}
Given $C_{n}(1,a_{2},\dots,a_{m})$ which is edge-transitive. Then there
exist $i,j$ (with $i$ possibly equal to $j$) such that $a_{i}a_{j}\equiv\pm1$
mod $n$.
\end{conjecture}

\begin{theorem}
Suppose $G=C_{n}(1,a_{2},\dots ,a_{m})$ and the set $\{1,a_{2},\dots
,a_{m},-1,-a_{2},\dots ,-a_{m}\}$ forms a group under multiplication mod $n$
(which is necessarily a subgroup of the group of units). Then $G$ is
edge-transitive.
\end{theorem}

\begin{proof}
Since the chord lengths and their negatives form a group, multiplying the
group by any element of the group permutes the elements of the group. Thus,
we can send, without loss of generality, any chord of length 1 to any chord
of length $a_{i}$ by the edge map $x\mapsto a_{i}x$.
\end{proof}

\section{Conclusion}

In Theorems \ref{r1} and \ref{r2}, we investigated biregular, bipartite
graphs of the form $(r,2)$. It would be an interesting, but challenging
problem to explore the family of $(r,k)$ and determine which graphs are
edge-transitive. An additional problem would be to determine the number of
non-isomorphic graphs of this form.

In Section 3, we investigated necessary and sufficient conditions for
hexavalent ($6$-regular) graphs to be edge-transitive. In addition to
proving the two stated conjectures it would be interesting to determine
necessary and conditions for circulant graphs in general to be
edge-transitive. We note that the difficulty will surely increase as
circulants with larger number of chords are investigated, since the number
of possible groups may increase. However it may be possible to solve the
posed problem when $n$ is prime or the product of two distinct prime numbers.

\bigskip

\section{Acknowledgements}

\noindent \newline
The authors are very grateful to Stanis\l {}aw Radziszowski for useful
discussion and for processing graph data on edge transitive graphs up to 20
vertices. Research was supported by National Science Foundation Research
Experiences for Undergraduates Site Award \#1659075.

\pagebreak

\begin{center}
\noindent {\LARGE Appendix }
\end{center}

{\normalsize \noindent \newline
}

{\normalsize \noindent \newline
\textbf{Notation} }

\begin{itemize}
\item {\normalsize $PM$ stands for "Perfect Matching" }

\item {\normalsize $\overline G$ stands for the complement graph of $G$ }

\item {\normalsize $K_{s,s, \dots, s}$ is the complete multipartite graph
with partite sets of size $s$ }

\item {\normalsize Graphs labelled NoncayleyTransitive}${\normalsize (n,k)}$%
{\normalsize \ come from Gordon Royle's list of non-Cayley vertex-transitive
graphs at: http://staffhome.ecm.uwa.edu.au/\~{0}0013890/trans/ }
\end{itemize}

{\normalsize \noindent \newline
For graphs with fewer than six vertices, all of the graphs are either
cycles, complete graphs, or star graphs. }

{\normalsize \noindent \newline
\textbf{Edge-Transitive Graphs on 6 Vertices} \medskip }

1. {\normalsize $K_{1,5}$ }

2. {\normalsize $C_{6}$ }

3. {\normalsize $K_{2,4}$ }

4. {\normalsize $K_{3,3}$ }

5. {\normalsize $C_{6}^{2}$ }

6. {\normalsize $K_{6}$ }

{\normalsize \noindent \newline
\textbf{Edge-Transitive Graphs on 7 Vertices} \medskip }

1. {\normalsize $K_{1,6}$ }

2. {\normalsize $C_{7}$ }

3. {\normalsize $K_{2,5}$ }

4. {\normalsize $K_{3,4}$ }

5. {\normalsize $K_{7}$ }

{\normalsize \noindent \newline
\textbf{Edge-Transitive Graphs on 8 Vertices} \medskip }

1. {\normalsize $K_{8}$ }

2. {\normalsize $K_{1,7}$ }

3. {\normalsize $K_{2,6}$ }

4. {\normalsize $K_{3,5}$ }

5. {\normalsize $K_{4,4}$ }

6. {\normalsize $C_{8}$ }

7. {\normalsize $C_{8}^{3}$ }

8. {\normalsize Cube (with 6 faces and 12 edges) }

{\normalsize \noindent \newline
\textbf{Edge-Transitive Graphs on 9 Vertices} \medskip }

1. {\normalsize $K_{9}$ }

2. {\normalsize $K_{1,8}$ }

3. {\normalsize $C_{9}$ }

4. {\normalsize $K_{2,7}$ }

5. {\normalsize $K_{3,6}$ }

6. {\normalsize $K_{4,5}$ }

7. {\normalsize $C_{3}\times C_{3}$ }

8. {\normalsize $K_{3,3,3}$ }

9. {\normalsize $(4,2)$ biregular subgraph of $K_{3,6}$ }

{\normalsize \noindent \newline
\textbf{Edge-Transitive Graphs on 10 Vertices} \medskip }

1. {\normalsize $K_{1,9}$ }

2. {\normalsize $C_{1}$}$_{{\normalsize 0}}${\normalsize \ }

3. {\normalsize $(3,2)$ biregular subgraph of $K_{4,6}$ }

4. {\normalsize $Petersen$ }

5. {\normalsize $K_{2,8}$ }

6. {\normalsize $K_{5,5}-PM$ }

7. {\normalsize $Wreath(5,2)$ }

8. {\normalsize $K_{3,7}$ }

9. {\normalsize $K_{4,6}$ }

10. {\normalsize $K_{5,5}$ }

11. {\normalsize $\overline{Petersen}$ }

12 .{\normalsize $K_{10}-PM$ }

13. {\normalsize $K_{12}$ }

{\normalsize \noindent \newline
\textbf{Edge-Transitive Graphs on 11 Vertices} \medskip }

1. {\normalsize $K_{1,10}$ }

2. {\normalsize $C_{11}$ }

3. {\normalsize $K_{2,9}$ }

4. {\normalsize $K_{3,8}$ }

5. {\normalsize $K_{4,7}$ }

6. {\normalsize $K_{5,6}$ }

7. {\normalsize $K_{11}$ }

{\normalsize \noindent \newline
\textbf{Edge-Transitive Graphs on 12 Vertices} \newline
There are 19 edge-transitive graphs on 12 vertices: \noindent \newline
\newline
1 graphs : $n=12$; $e=11$; $K_{1,11}$ \newline
1 graphs : $n=12$; $e=12$; $C_{12}$ \newline
1 graphs : $n=12$; $e=16$; $(4,2)$ subgraph of $K_{4,8}$ \newline
1 graphs : $n=12$; $e=18$; $(6,2)$ subgraph of $K_{3,9}$ \newline
1 graphs : $n=12$; $e=20$; $K_{2,10}$ \newline
3 graphs : $n=12$; $e=24$; $(3,6)$ subgraph of $K_{8,4}$; Wreath(6,2);
Cuboctahedral \newline
1 graphs : $n=12$; $e=27$; $K_{3,9}$ \newline
2 graphs : $n=12$; $e=30$; $K_{6,6}-PM$; Icosahedral \newline
1 graphs : $n=12$; $e=32$; $K_{4,8}$ \newline
1 graphs : $n=12$; $e=35$; $K_{5,7}$ \newline
2 graphs : $n=12$; $e=36$; $K_{6,6}$; $\overline{K_{3}\times K_{4}}$ \newline
1 graphs : $n=12$; $e=48$; $K_{4,4,4}$ \newline
1 graphs : $n=12$; $e=54$; $K_{3,3,3,3}$ \newline
1 graphs : $n=12$; $e=60$; $K_{12}-PM$ \newline
1 graphs : $n=12$; $e=66$; $K_{12}$ }

{\normalsize \noindent }

{\normalsize \noindent \newline
\textbf{Edge-Transitive Graphs on 13 Vertices} \newline
There are 10 edge-transitive graphs on 13 vertices: \noindent \newline
\newline
1 graphs : $n=13$; $e=12$; $K_{1,12}$ \newline
1 graphs : $n=13$; $e=13$; $C_{13}$ \newline
1 graphs : $n=13$; $e=22$; $K_{2,11}$ \newline
1 graphs : $n=13$; $e=26$; $C_{13}(1,5)$ \newline
1 graphs : $n=13$; $e=30$; $K_{3,10}$ \newline
1 graphs : $n=13$; $e=36$; $K_{4,9}$ \newline
1 graphs : $n=13$; $e=39$; Paley(13) \newline
1 graphs : $n=13$; $e=40$; $K_{5,8}$ \newline
1 graphs : $n=13$; $e=42$; $K_{6,7}$ \newline
1 graphs : $n=13$; $e=78$; $K_{13}$ }

{\normalsize \noindent \newline
}

{\normalsize \noindent \textbf{Edge-Transitive Graphs on 14 Vertices} 
\newline
There are 16 edge-transitive graphs on 14 vertices: \noindent \newline
\newline
1 graphs : $n=14$; $e=13$; $K_{1,13}$ \newline
1 graphs : $n=14$; $e=14$; $C_{14}$ \newline
1 graphs : $n=14$; $e=21$; Heawood $=(3,3)$ subgraph of $K_{7,7}$ \newline
3 graphs : $n=14$; $e=24$; $K_{2,12}$; $2$ $(4,3)$ biregular subgraphs of $%
K_{6,8}$ \newline
2 graphs : $n=14$; $e=28$; Wreath(7,2); $(4,4)$ subgraph of $K_{7,7}$ 
\newline
1 graphs : $n=14$; $e=33$; $K_{3,11}$ \newline
1 graphs : $n=14$; $e=40$; $K_{4,10}$ \newline
1 graphs : $n=14$; $e=42$; $K_{7,7}-PM$ \newline
1 graphs : $n=14$; $e=45$; $K_{5,9}$ \newline
1 graphs : $n=14$; $e=48$; $K_{6,8}$ \newline
1 graphs : $n=14$; $e=49$; $K_{7,7}$ \newline
1 graphs : $n=14$; $e=84$; $K_{14}-PM$ \newline
1 graphs : $n=14$; $e=91$; $K_{14}$ }

{\normalsize \noindent \newline
\textbf{Edge-Transitive Graphs on 15 Vertices} \newline
There are 25 edge-transitive (connected) graphs on 15 vertices. \noindent 
\newline
\newline
1 graphs : $n=15$; $e=14$; $K_{1,14}$ \newline
1 graphs : $n=15$; $e=15$; $C_{15}$ \newline
1 graphs : $n=15$; $e=18$; $(3,2)$ subgraph of $K_{6,9}$ \newline
2 graphs : $n=15$; $e=20$; 2 $(4,2)$ subgraphs of $K_{5,10}$ \newline
1 graphs : $n=15$; $e=24$; $(8,2)$subgraph of $K_{3,12}$ \newline
1 graphs : $n=15$; $e=26$; $K_{2,13}$ \newline
3 graphs : $n=15$; $e=30$; $(6,3)$ subgraph of $K_{5,10}$; $L$(Petersen); $%
C_{15}(1,4)$ \newline
3 graphs : $n=15$; $e=36$; $K_{3,12}$; 2 $(6,4)$ subgraphs of $K_{6,9}$ 
\newline
1 graphs : $n=15$; $e=40$; $(8,4)$ subgraph of $K_{5,10}$ \newline
1 graphs : $n=15$; $e=44$; $K_{4,11}$ \newline
2 graphs : $n=15$; $e=45$; $(6,2)$ Kneser graph; Trpl$(C_{5})$ \newline
1 graphs : $n=15$; $e=50$; $K_{5,10}$ \newline
1 graphs : $n=15$; $e=54$; $K_{6,9}$ \newline
1 graphs : $n=15$; $e=56$; $K_{7,8}$ \newline
2 graphs : $n=15$; $e=60$; $\overline{K_{3}\times K_{5}}$; (6,2) Johnson
graph \newline
1 graphs : $n=15$; $e=75$; $K_{5,5,5}$ \newline
1 graphs : $n=15$; $e=90$; $K_{3,3,3,3,3}$ \newline
1 graphs : $n=15$; $e=105$; $K_{15}$ }

{\normalsize \noindent \newline
\textbf{Edge-Transitive Graphs on 16 Vertices} \newline
There are 26 edge-transitive (connected) graphs on 16 vertices: \noindent 
\newline
\newline
1 graphs : $n=16$; $e=15$; $K_{1,15}$ \newline
1 graphs : $n=16$; $e=16$; $C_{16}$ \newline
3 graphs : $n=16$; $e=24$; M\"{o}bius-Kantor graph; 2 $(6,2)$ subgraphs of $%
K_{4,12}$ \newline
1 graphs : $n=16$; $e=28$; $K_{2,14}$ \newline
1 graphs : $n=16$; $e=30$; $(5,3)$ subgraph of $K_{6,10}$ \newline
2 graphs : $n=16$; $e=32$; $Q_{4}$; Wreath(8,2) \newline
1 graphs : $n=16$; $e=36$; $(9,3)$ subgraph of $K_{4,12}$ \newline
1 graphs : $n=16$; $e=39$; $K_{3,13}$ \newline
1 graphs : $n=16$; $e=40$; Clebsch graph \newline
4 graphs : $n=16$; $e=48$; $K_{4,12}$; Shrikhande graph; ${K_{4}\times K_{4}}
$; $Haar(187)$ \newline
1 graphs : $n=16$; $e=55$; $K_{5,11}$ \newline
1 graphs : $n=16$; $e=56$; $K_{8,8}-PM\cong \overline{K_{8}\times K_{2}}$ 
\newline
1 graphs : $n=16$; $e=60$; $K_{6,10}$ \newline
1 graphs : $n=16$; $e=63$; $K_{7,9}$ \newline
1 graphs : $n=16$; $e=64$; $K_{8,8}$ \newline
1 graphs : $n=16$; $e=72$; $\overline{K_{4}\times K_{4}}$ \newline
1 graphs : $n=16$; $e=80$; Complement of Clebsch \newline
1 graphs : $n=16$; $e=96$; $K_{4,4,4,4}$ \newline
1 graphs : $n=16$; $e=112$; $K_{16}-PM$ \newline
1 graphs : $n=16$; $e=120$; $K_{16}$ }

{\normalsize \noindent \newline
\textbf{Edge-Transitive Graphs on 17 Vertices} \newline
There are 12 edge-transitive (connected) graphs on 17 vertices: \noindent 
\newline
\newline
1 graphs : $n=17$; $e=16$; $K_{1,16}$ \newline
1 graphs : $n=17$; $e=17$; $C_{17}$ \newline
1 graphs : $n=17$; $e=30$; $K_{2,15}$ \newline
1 graphs : $n=17$; $e=34$; $C_{17}(1,4)$ \newline
1 graphs : $n=17$; $e=42$; $K_{3,14}$ \newline
1 graphs : $n=17$; $e=52$; $K_{4,13}$ \newline
1 graphs : $n=17$; $e=60$; $K_{5,12}$ \newline
1 graphs : $n=17$; $e=66$; $K_{6,11}$ \newline
1 graphs : $n=17$; $e=68$; Paley$(17)$ \newline
1 graphs : $n=17$; $e=70$; $K_{7,10}$ \newline
1 graphs : $n=17$; $e=72$; $K_{8,9}$ \newline
1 graphs : $n=17$; $e=136$; $K_{17}$ }

{\normalsize \noindent \newline
\textbf{Edge-Transitive Graphs on 18 Vertices} \newline
There are 28 edge-transitive (connected) graphs on 18 vertices: \noindent 
\newline
\newline
1 graphs : $n=18$; $e=17$; $K_{1,17}$ \newline
1 graphs : $n=18$; $e=18$; $C_{18}$ \newline
2 graphs : $n=18$; $e=24$; 2 $(4,2)$ subgraphs of $K_{6,12}$ \newline
1 graphs : $n=18$; $e=27$; Pappus graph \newline
1 graphs : $n=18$; $e=30$; $(10,2)$ subgraph of $K_{3,15}$ \newline
1 graphs : $n=18$; $e=32$; $K_{2,16}$ \newline
3 graphs : $n=18$; $e=36$; $(6,3)$ subgraph of $K_{6,12}$; $(4,4)$ subgraph
of $K_{9,9}$, Wreath$(9,2)$ \newline
1 graphs : $n=18$; $e=45$; $K_{3,15}$ \newline
1 graphs : $n=18$; $e=48$; $(8,4)$ subgraph of $K_{6,12}$ \newline
2 graphs : $n=18$; $e=54$; 2 $(6,6)$ subgraphs of $K_{9,9}$ \newline
1 graphs : $n=18$; $e=56$; $K_{4,14}$ \newline
1 graphs : $n=18$; $e=60$; $(10,5)$ subgraph of $K_{6,12}$ \newline
1 graphs : $n=18$; $e=65$; $K_{5,13}$ \newline
3 graphs : $n=18$; $e=72$; $K_{9,9}-PM$ $\cong \overline{K_{9}\times K_{2}}$%
; $K_{6,12}$; $H_{1}$ \newline
1 graphs : $n=18$; $e=77$; $K_{7,11}$ \newline
1 graphs : $n=18$; $e=80$; $K_{8,10}$ \newline
1 graphs : $n=18$; $e=81$; $K_{9,9}$ \newline
1 graphs : $n=18$; $e=90$; $\overline{K_{6}\times K_{3}}$ \newline
1 graphs : $n=18$; $e=108$; $K_{6,6,6}$ \newline
1 graphs : $n=18$; $e=135$; $K_{3,3,3,3,3,3}$ \newline
1 graphs : $n=18$; $e=144$; $K_{18}$ - PM \newline
1 graphs : $n=18$; $e=153$; $K_{18}$ }

\FRAME{dtbpFU}{1.6682in}{1.5861in}{0pt}{\Qcb{Figure 1. The graph $H_{1}$ on
18 vertices.}}{}{Figure}{\special{language "Scientific Word";type
"GRAPHIC";maintain-aspect-ratio TRUE;display "USEDEF";valid_file "T";width
1.6682in;height 1.5861in;depth 0pt;original-width 4.6873in;original-height
4.4512in;cropleft "0";croptop "1";cropright "1";cropbottom "0";tempfilename
'OW9TAU06.wmf';tempfile-properties "XPR";}}

{\normalsize \noindent \newline
\textbf{Edge-Transitive Graphs on 19 Vertices} \newline
There are 12 edge-transitive (connected) graphs on 19 vertices: \noindent 
\newline
\newline
1 graphs : $n=19$; $e=18$; $K_{1,18}$ \newline
1 graphs : $n=19$; $e=19$; $C_{19}$ \newline
1 graphs : $n=19$; $e=34$; $K_{2,17}$ \newline
1 graphs : $n=19$; $e=48$; $K_{3,16}$ \newline
1 graphs : $n=19$; $e=57$; $C_{19}(1,7,8)$ \newline
1 graphs : $n=19$; $e=60$; $K_{4,15}$ \newline
1 graphs : $n=19$; $e=70$; $K_{5,14}$ \newline
1 graphs : $n=19$; $e=78$; $K_{6,13}$ \newline
1 graphs : $n=19$; $e=84$; $K_{7,12}$ \newline
1 graphs : $n=19$; $e=88$; $K_{8,11}$ \newline
1 graphs : $n=19$; $e=90$; $K_{9,10}$ \newline
1 graphs : $n=19$; $e=171$; $K_{19}$ }

{\normalsize \noindent \newline
\textbf{Edge-Transitive Graphs on 20 Vertices} \newline
There are 43 edge-transitive (connected) graphs on 20 vertices. \noindent 
\newline
\newline
1 graphs : $n=20$; $e=19$; $K_{1,19}$ \newline
1 graphs : $n=20$; $e=20$; $C_{20}$ \newline
1 graphs : $n=20$; $e=24$; $(3,2)$ subgraph of $K_{8,12}=$1-Menger sponge
graph \newline
3 graphs : $n=20$; $e=30$; $(6,2)$ subgraph of $K_{5,15}$; Desargues Graph;
Dodecahedral graph \newline
1 graphs : $n=20$; $e=32$; $(8,2)$ subgraph of $K_{4,16}$ \newline
1 graphs : $n=20$; $e=36$; $K_{2,18}$ \newline
4 graphs : $n=20$; $e=40$; Folkman; $Wreath(10,2)$; $Haar(525)$;
NoncayleyTransitive(20,4) \newline
5 graphs : $n=20$; $e=48$; 4 $(6,4)$ subgraphs of $K_{8,12}$; 1 $(12,3)$
subgraph of $K_{4,16}$ \newline
1 graphs : $n=20$; $e=51$; $K_{3,17}$ \newline
7 graphs : $n=20$; $e=60$; 1 $(12,4)$ subgraph of $K_{5,15}$; 2 $(6,6)$
subgraphs of $K_{10,10}$ (one is NoncayleyTransitive(20,12); $C_{20}(1,6,9)$%
; $G_{1}$; $G_{2}$; $(5,2)$-arrangement graph \newline
1 graphs : $n=20$; $e=64$; $K_{4,16}$ \newline
1 graphs : $n=20$; $e=72$; $(9,6)$ subgraph of $K_{8,12}$ \newline
1 graphs : $n=20$; $e=75$; $K_{5,15}$ \newline
2 graphs : $n=20$; $e=80$; $(8,8)$ subgraph of $K_{10,10}$; $Wreath(5,4)$ 
\newline
1 graphs : $n=20$; $e=84$; $K_{6,14}$ \newline
2 graphs : $n=20$; $e=90$; $K_{10,10}-PM$; 6-tetrahedral (Johnson) graph 
\newline
1 graphs : $n=20$; $e=91$; $K_{7,13}$ \newline
1 graphs : $n=20$; $e=96$; $K_{8,12}$ \newline
1 graphs : $n=20$; $e=99$; $K_{9,11}$ \newline
1 graphs : $n=20$; $e=100$; $K_{10,10}$ \newline
2 graphs : $n=20$; $e=120$; $\overline{K_{5}\times K_{4}}$; $G_{3}$ \newline
1 graphs : $n=20$; $e=150$; $K_{5,5,5,5}$ \newline
1 graphs : $n=20$; $e=160$; $K_{4,4,4,4,4}$ \newline
1 graphs : $n=20$; $e=180$; $K_{20}-PM$ \newline
1 graphs : $n=20$; $e=190$; $K_{20}$ }

\FRAME{dtbpFU}{4.2488in}{1.4442in}{0pt}{\Qcb{Figure 2. The graphs $G_{1}$, $%
G_{2}$, and $G_{3}$.}}{}{Figure}{\special{language "Scientific Word";type
"GRAPHIC";maintain-aspect-ratio TRUE;display "USEDEF";valid_file "T";width
4.2488in;height 1.4442in;depth 0pt;original-width 8.278in;original-height
2.7778in;cropleft "0";croptop "1";cropright "1";cropbottom "0";tempfilename
'OW9TAU07.wmf';tempfile-properties "XPR";}}

\end{document}